\documentclass[a4paper,reqno]{amsart}

\usepackage{amsmath,amsfonts,amsthm,amssymb,hyperref,latexsym,mathrsfs,enumerate,pgfplots,tikz-cd}
\pgfplotsset{compat=1.16}

\DeclareMathOperator{\diag}{\mathrm{diag}}

\DeclareMathOperator{\tr}{\mathrm{tr}}

\DeclareMathOperator{\diam}{\mathrm{diam}}

\DeclareMathOperator{\id}{\mathrm{id}}

\DeclareMathOperator{\ad}{\mathrm{Ad}}
\DeclareMathOperator{\cu}{\mathrm{Cu}}

\DeclareMathOperator{\cs}{\mathrm{C}^*}
\DeclareMathOperator{\js}{\mathcal{Z}}
\DeclareMathOperator{\Ann}{\mathrm{Ann}}
\DeclareMathOperator{\Aut}{\mathrm{Aut}}
\DeclareMathOperator{\End}{\mathrm{End}}
\DeclareMathOperator{\Ell}{\mathrm{Ell}}
\DeclareMathOperator{\Inv}{\mathrm{Inv}}
\DeclareMathOperator{\hol}{\mathrm{H\ddot{o}l}}
\DeclareMathOperator{\aff}{\mathrm{Aff}}

\DeclareMathOperator{\lip}{\mathrm{Lip}}

\newcommand{\nn}{\mathbb{N}}
\newcommand{\zz}{\mathbb{Z}}
\newcommand{\qq}{\mathbb{Q}}
\newcommand{\rr}{\mathbb{R}}
\newcommand{\cc}{\mathbb{C}}
\newcommand{\tot}{\mathbb{T}}

\makeatletter
\@namedef{subjclassname@2020}{%
  \textup{2020} Mathematics Subject Classification}
\makeatother

\newtheorem {theorem}{Theorem}[section]
\newtheorem {lemma}[theorem]{Lemma}
\newtheorem {proposition}[theorem]{Proposition}

\newtheorem {question}{Question}

\newtheorem {thm}{Theorem}

\theoremstyle {definition}

\newtheorem {example}[theorem]{Example}
\newtheorem {definition}[theorem]{Definition}
\newtheorem {remark}[theorem]{Remark}

\numberwithin{equation}{section}

\title{Chaotic tracial dynamics}

\author[B.~Jacelon]{Bhishan Jacelon}
\address[B.~Jacelon]{
Institute of Mathematics of the Czech Academy of Sciences\\ \v{Z}itn\'{a} 25\\115 67 Prague 1\\Czech Republic}
\email{bjacelon@gmail.com}


\begin{document}

\begin{abstract}
The action on the trace space induced by a generic automorphism of a suitable finite classifiable $\cs$-algebra is shown to be chaotic and weakly mixing. Model $\cs$-algebras are constructed to observe the central limit theorem and other statistical features of strongly chaotic tracial actions. Genericity of finite Rokhlin dimension is used to describe $KK$-contractible stably projectionless $\cs$-algebras as crossed products. 
\end{abstract}

\maketitle

\section{Introduction} \label{section:intro}

This article is an investigation of the tracial properties of endomorphisms of classifiable $\cs$-algebras. Here,  `classifiable' means specifically via the \emph{Elliott invariant}
\[
\Ell(A) = (K_0(A),K_0(A)_+,\Sigma(A),T(A),\rho_A\colon K_0(A)\to\aff(T(A)),K_1(A)),
\]
whose components are $K$-theory $K_*(A)$ (together with its order structure $K_0(A)_+$ and scale $\Sigma(A)$, which in the unital case is simply the $K_0$-class of the unit), the tracial state space $T(A)$, and the pairing $\rho_A$ between the two. By \cite{Elliott:2016ab,Gong:2020ud,Gong:2020uf,Elliott:2020wc,Gong:2020vg,Gong:2021tw} (see also \cite{Carrion:wz}), infinite-dimensional, simple, separable $\cs$-algebras that have continuous scale, finite nuclear dimension (which by \cite{Castillejos:2020wv,Castillejos:2021wm,Tikuisis:2012kx,Winter:2012pi} is equivalent to tensorial absorption of the Jiang--Su algebra $\js$) and satisfy the universal coefficient theorem (UCT), are classifiable. And by \cite{Toms:2008vn}, a finer invariant than $\Ell$ is needed to extend classification beyond this class.

In this paper, we consider both unital and nonunital classifiable $\cs$-algebras. `Continuous scale' is a technical assumption that entails algebraic simplicity of $A$ and compactness of $T(A)$ (see \cite[\S5]{Elliott:2020vp}). It is automatic for a separable, unital $\cs$-algebra and in the nonunital case it in particular ensures that the algebra does not support any unbounded lower semicontinuous traces. While the classification presented in \cite{Elliott:2020wc,Gong:2020vg,Gong:2021tw} goes beyond $\cs$-algebras with continuous scale, it is a necessary assumption for the analysis undertaken in the sequel.

An important problem in the study of group actions on simple, separable, $\js$-stable $\cs$-algebras is to determine when a given action $\alpha\colon G\to\Aut(A)$ is cocycle conjugate to its tensor product $\alpha\otimes\id_{\js}$ with the trivial action on $\js$. (As discussed in the introduction of \cite{Wouters:2021uw}, these are the actions one can expect to classify.) In \cite{Gardella:2021tb}, the problem is solved for actions of discrete, countable, amenable groups $G$ on unital $\cs$-algebras $A$ under the assumptions that the extreme boundary $\partial_e(T(A))$ of $T(A)$ is compact and of finite covering dimension and that the action of $G$ on $\partial_e(T(A))$ has finite orbits of bounded size, with Hausdorff orbit space $\partial_e(T(A))/G$. This motivates us to pose the following.

\begin{question} \label{q1}
What is the generic tracial behaviour of an automorphism of a stably finite classifiable $\cs$-algebra?
\end{question}

In our context, stable finiteness is equivalent to the trace space $T(A)$ being nonempty (see \cite[\S1.1.3]{Rordam:2002yu}). By `generic tracial behaviour' we mean that we seek to identify properties of induced affine homeomorphisms $T(A)\to T(A)$ that hold residually, that is, for at least a dense $G_\delta$ set of automorphisms $A\to A$ (in the topology of pointwise convergence). We will address Question~\ref{q1} under the additional assumptions that:
\begin{itemize}
\item the extreme boundary $\partial_e(T(A))$ of $T(A)$ is a finite-dimensional, compact, connected topological manifold;
\item the automorphism $\alpha$ of $A$ fixes an Oxtoby--Ulam \emph{OU trace} (roughly, a noncommutative analogue of Lebesgue measure---see \S~\ref{subsection:chaos} and \S~\ref{ssc:rep});
\item the pairing $\rho_A\colon K_0(A)\to\aff(T(A))$ is trivial;
\item $K_1(A)$ is torsion free.
\end{itemize}

With these hypotheses, we are able to use classification to lift known results about topological dynamical systems to the $\cs$-level. The strategy is simple (as long as we are willing to make free use of heavy classification machinery): Given an endomorphism $\alpha$, nudge its action on $\partial_e(T(A))$ to be in a desired topological class, then lift the perturbed invariant to an endomorphism $\beta$ which, by virtue of the closeness of $\Ell(\beta)$ to $\Ell(\alpha)$, is (up to conjugation by a unitary) within a given $\varepsilon$ of $\alpha$ on a given finite set. Actually, the invariant we must use is not $\Ell$ but one that also includes `total $K$-theory' and `Hausdorffised algebraic $K_1$' (see the proof of Theorem~\ref{lemma:denselift} for a brief primer), together with suitable pairings. The above assumptions on $K$-theory ensure that the compatibility demanded by these pairings is automatically satisfied.

In short, the following theorem indicates that when there are infinitely many extremal traces, the typical situation can be very different to the one considered in \cite{Gardella:2021tb}. (Note, however, that it does not represent an obstruction to generalising \cite{Gardella:2021tb} for single automorphisms; in the breakthrough article \cite{Wouters:2021uw}, it is shown how to treat the portion of $\partial_e(T(A))$ with unbounded orbits.)

\begin{thm} \label{thm1}
For every OU trace $\tau$ on $A$, the generic $\tau$-preserving automorphism of $A$ induces a homeomorphism of $\partial_e(T(A))$ that is weakly mixing and is moreover Devaney-chaotic if $\dim \partial_e(T(A)) \ge 2$.
\end{thm}

See \S~\ref{section:topology} for the definitions of chaos and mixing, which are properties that are satisfied by, for example, hyperbolic toral automorphisms like Arnold's cat map. A particular consequence for the algebras covered by Theorem~\ref{thm1} is that, in contrast to the case of finitely many extremal traces, a typical tracial property of automorphisms is \emph{sensitive dependence on initial conditions}: Arbitrarily close extremal traces eventually get moved far apart. 

Whereas \emph{weak} mixing is residual, in both the measure-theoretic \cite{Halmos:1944wm} and topological \cite{Katok:1970wx, Alpern:1979th} settings, \emph{strong} mixing is enjoyed by meagre sets of measure-preserving transformations \cite{Rohlin:1948wr}. On the other hand, many interesting dynamical systems, for example those associated to Anosov diffeomorphisms \cite[\S1E]{Bowen:2008uh}, are not just strongly mixing but have exponentially fast mixing rates. This chaotic behaviour is reflected in such statistical features as exponential decay of correlations (EDC) and the central limit theorem (CLT). Once again, classification will afford us $\cs$-interpretations of these phenomena.

We will, however, need some additional structure in the ambient $\cs$-algebra to have satisfactory versions of the CLT and EDC, as these require not just continuous but Lipschitz (or at least H\"{o}lder continuous) observables. To that end, for a given compact, connected metric space $(X,d)$, we construct in Theorem~\ref{thm:models} an approximately subhomogeneous $\cs$-algebra $A$ with $\partial_e(T(A))\cong X$ and for which $\{a\mid \hat a \in \lip(X,d)\}$ is dense in the set $A_{sa}$ of self-adjoint elements of $A$ (where $\hat a\colon X\to\mathbb{R}$, which we refer to as an \emph{observable of the tracial dynamics}, denotes the evaluation map $\tau\mapsto\tau(a)$). Theorem~\ref{thm:lift} then provides the means of lifting a given dynamical system on $X$ to the $\cs$-level. In the uniquely ergodic but not necessarily chaotic setting, following \cite{Benoist:2016ur} (see Proposition~\ref{prop:breiman}) this structure will allow for uniform estimates of large deviation, that is, finite-time estimates of the rate of the tracial convergence guaranteed by Birkhoff's ergodic theorem. 

The models of the previous paragraph will in general have complicated $K_1$-groups. This is potentially vexing if one is interested in managing, via the Pimsner--Voiculescu sequence \cite{Pimsner:1980yu}, the $K$-theory of the crossed product of $A$ by an automorphism $\alpha\colon A\to A$ constructed to witness a given topological dynamical system $h\colon X\to X$. This issue can be addressed by lifting $h$ not to an automorphism but to a $K_1$-killing endomorphism, which, as in \cite{Stacey:1993uc}, can be extended to an automorphism of a stably isomorphic $\cs$-algebra. (Endomorphisms also allow for the inclusion of noninvertible tracial dynamics $h\colon X\to X$, for example the Pomeau--Manneville-type system described in \cite[\S5]{Jacelon:2021vc}, and to appeal to and interpret existing results \cite{Bobok:2020uj} about such systems.) In particular, if $X$ is an odd sphere, then the extended action is on an inductive limit of prime dimension drop algebras; upon computing the Elliott invariant (see Example~\ref{example:dps} and Remark~\ref{remark:final}), one learns via classification that these crossed products are just different descriptions of the ones considered in \cite{Deeley:aa}. 

The above discussion of the range of the Elliott invariant of crossed products motivates our second question.

\begin{question} \label{q2}
Which classifiable $\cs$-algebras can be described as crossed products of stably finite classifiable $\cs$-algebras by the integers?
\end{question}

Since Kirchberg algebras are already known to be included in the answer to Question~\ref{q2} (see, for example, \cite[Proposition 4.3.3]{Rordam:2002yu}), we focus our attention on stably finite targets. Combining work of Downarowicz \cite{Downarowicz:1991te} and Szab\'{o}, Wu and Zacharias \cite{Szabo:2019te}, we show that every metrisable Choquet simplex is attainable as the tracial state space of a crossed product of a classifiable $\cs$-algebra $A$, for which $\partial_e(T(A))$ is a Cantor space, by an automorphism with finite Rokhlin dimension. In the $KK$-contractible setting (or equivalently under the UCT, assuming that $K_*(A)=0$), \cite{Elliott:2020wc} then shows that, up to stable isomorphism, the full class is exhausted.

\begin{thm} \label{thm3}
Let $\mathcal{K}$ be the class of infinite-dimensional, simple, separable, $KK$-contractible $\cs$-algebras that satisfy the UCT and have continuous scale, nonempty trace space and finite nuclear dimension. Then, for every $B\in\mathcal{K}$ there exists $A\in\mathcal{K}$ and an automorphism $\alpha\in\Aut(A)$ such that $\partial_e(T(A))$ is compact and zero dimensional, and $A\rtimes_\alpha\zz\cong B$.
\end{thm}

This article is organised as follows. In \S~\ref{section:topology}, we collate relevant properties of measure-preserving topological dynamical systems. Then in \S~\ref{section:classification}, we explain how to use classification to lift these properties to statements about $\cs$-dynamical systems. Examples of strongly chaotic dynamics, descriptions of associated statistical features like the CLT and EDC and the construction of models to suitably witness these phenomena are described in \S~\ref{section:strongchaos}. Finally, in \S~\ref{section:range} we show how to obtain $KK$-contractible classifiable $\cs$-algebras as crossed products.

\subsection*{Acknowledgements} This research was supported by the GA\v{C}R project 20-17488Y and RVO: 67985840. I am grateful to Ali Asadi--Vasfi, Karen Strung, Andrea Vaccaro, Stuart White and Huaxin Lin for many helpful discussions and to Cape Breton University for giving me the opportunity to teach its statistics courses. I would also like to thank the anonymous referee whose suggestions helped improve the clarity of the article.

\section{Generic properties of topological dynamical systems} \label{section:topology}

Chaos and mixing will be familiar notions to those working in dynamical systems but perhaps not to $\cs$-algebraists. We briefly introduce these concepts for the reader's convenience.

\subsection{Chaos} \label{subsection:chaos}

\begin{definition} \label{def:chaos}
Let $(X,d)$ be an infinite metric space and $h\colon X\to X$ a continuous map. The dynamical system $(X,h)$ is said to be \emph{chaotic} in the sense of Devaney \cite[Definition 8.5]{Devaney:1989ug} if:
\begin{enumerate}[1.]
\item \label{item:dev1} for every nonempty open sets $U,V\subseteq X$, there exists $n\in\nn$ such that $h^n(U)\cap V\ne\emptyset$ (that is, $(X,h)$ is \emph{transitive} or equivalently, \emph{irreducible});
\item \label{item:dev2} the periodic points of $h$ are dense in $X$;
\item \label{item:dev3} the system has \emph{sensitive dependence on initial conditions}, that is, there exists $\delta>0$ such that, for every $x\in X$ and every $\varepsilon>0$, there exist $y\in X$ and $n\in\nn$ such that $d(x,y)<\varepsilon$ and $d(h^nx,h^ny)\ge\delta$.
\end{enumerate}
\end{definition}
As pointed out in \cite{Banks:1992ux}, (\ref{item:dev3}) actually follows from (\ref{item:dev1}) and (\ref{item:dev2}) (and if $X$ is compact, it is easy to see directly that (\ref{item:dev3}) is preserved under topological conjugation). This means that this notion of chaos is genuinely topological. An equivalent characterisation is given in \cite{Touhey:1997wm}: $(X,h)$ is chaotic if and only if, for every nonempty open sets $U,V\subseteq X$, there is a periodic point whose forward orbit intersects both $U$ and $V$.

An important class of Devaney-chaotic dynamical systems is provided by the irreducible components of nonwandering Smale spaces. (See \cite[Theorem 3.5]{Bowen:2008uh} and \cite[3.8]{Bowen:2008uh}. Note that these results, though stated for diffeomorphisms, apply just as well to the abstract setting of Smale spaces; see \cite[Chapter 7]{Ruelle:2004ul}.) We will see some more examples in \S~\ref{section:strongchaos}, but in fact, under certain circumstances that we now recall (Theorem~\ref{thm:chaoticman} below), Devaney-chaotic maps are generic among measure-preserving homeomorphisms of topological manifolds.

The manifolds discussed in this paper will always be compact and will often be connected. Measures will typically be particularly tractable.

\begin{definition}[\cite{Alpern:2000ux}] \label{def:ou}
A Borel probability measure $\mu$ on a topological manifold $X$ is called an \emph{Oxtoby--Ulam (OU) measure} if it is faithful, nonatomic and zero on the boundary $\partial X$ of $X$ (if there is one).
\end{definition}

When $X=[0,1]^n$, OU measures are precisely those that are homeomorphic images of Lebesgue measure (see \cite[Theorem 2]{Oxtoby:1941aa}), and when $X$ is boundaryless, OU measures are generic (see \cite[Proposition 1.4]{Fathi:1980aa}).

We equip the set $\mathcal{H}(X,\mu)$ (respectively, $\mathcal{C}(X,\mu)$) of $\mu$-preserving homeomorphisms (respectively, continuous maps) $X\to X$ with the compact open topology, which for compact metric spaces $X$ means the topology of uniform convergence. A complete metric for the topology on $\mathcal{H}(X,\mu)$ is
\begin{equation} \label{eqn:polish1}
\rho(g,h) = \sup_{x\in X}d(g(x),h(x))+d(g^{-1}(x),h^{-1}(x)).
\end{equation}

\begin{theorem}  [\cite{Aarts:1999wc, Daalderop:2000wm}] \label{thm:chaoticman}
Let $X$ be a compact topological manifold of dimension $n\ge 2$ and let $\mu$ be an OU measure on $X$. Then, the set of Devaney-chaotic elements of $\mathcal{H}(X,\mu)$ is residual (that is, contains a dense $G_\delta$ set).
\end{theorem}

The density part of Theorem~\ref{thm:chaoticman} is also a corollary of \cite[Lemma 2]{Alpern:1999vi}, which shows how to perturb a measure-preserving homeomorphism of the cube to one that cyclically permutes arbitrarily small dyadic cubes (and is therefore chaotic). As mentioned in \cite{Alpern:1999vi}, this can be translated to the setting of compact, connected manifolds via a modification of a theorem of Brown (see \cite[Theorem 9.6]{Alpern:2000ux}). Moreover, Theorem~\ref{thm:chaoticman} still holds if condition (\ref{item:dev3}) is strengthened to a truly metric property called `maximal' dependence on initial conditions (see \cite[Theorem 4.8]{Alpern:2000ux}).

\subsection{Mixing} \label{subsection:mixing}

\begin{definition} \label{def:mixing}
Let $(X,\Sigma,\mu)$ be a probability space. A $\mu$-preserving measurable map $h\colon X\to X$ is:
\begin{enumerate}[1.]
\item \emph{antiperiodic} if the set of periodic points of $h$ has measure zero;
\item \emph{ergodic} if, for every $A,B\in\Sigma$,
\[
\lim_{n\to\infty}\frac{1}{n}\sum_{k=0}^{n-1}\mu(h^{-k}(A)\cap B) = \mu(A)\mu(B);
\]
\item \emph{weakly mixing} if, for every $A,B\in\Sigma$,
\[
\lim_{n\to\infty}\frac{1}{n}\sum_{k=0}^{n-1}|\mu(h^{-k}(A)\cap B) - \mu(A)\mu(B)| = 0;
\]
\item \emph{strongly mixing} if, for every $A,B\in\Sigma$,
\[
\lim_{n\to\infty}\mu(h^{-n}(A)\cap B) = \mu(A)\mu(B).
\]
\end{enumerate}
\end{definition}

\begin{remark} \label{remark:mixing}
\begin{enumerate}[1.]
\item That this formulation of ergodicity is implied by the usual one (invariant sets having trivial measure) is an application of Birkhoff's ergodic theorem (see \cite[`Consequences of ergodicity']{Halmos:1960vh}). The other direction is easily deduced upon taking $A=B^c$.
\item It is also immediate that strong mixing implies weak mixing, which implies ergodicity, which implies antiperiodicity. In fact (see \cite[`Mixing']{Halmos:1960vh}) an invertible transformation $h$ is weakly mixing if and only if $h\times h$ is ergodic, which holds if and only if the only eigenvalue of the unitary operator $U_h\in\mathcal{B}(L^2(X,\Sigma,\mu))$, $f\mapsto f\circ h$, is $1$ with the constants as the only eigenfunctions.
\item If $X$ is a topological space and $\Sigma$ is its Borel $\sigma$-algebra, then a strongly mixing continuous map $h$ is topologically mixing on the support of $\mu$.
\item If $(X,h)$ is a mixing axiom A diffeomorphism, then the `Bowen measure' is strongly mixing (see, for example, \cite[Theorem 4.1, \S1E]{Bowen:2008uh}).
\end{enumerate}
\end{remark}
The \emph{measure algebra} associated to $(X,\Sigma,\mu)$ is the Boolean algebra $\mathfrak{B}$ consisting of the equivalence classes of measurable sets (where equivalence of $E$ and $F$ means that their symmetric difference has measure zero) and equipped with the usual set operations (intersection, union, complementation). If $\mu$ is nonatomic and $\mathfrak{B}$ is separable (with respect to the metric $d(E,F)=E \triangle F$), then $\mathfrak{B}$ is isomorphic to the measure algebra of the unit interval $I$ with Lebesgue measure $\lambda$ (see, for example, \cite[Theorem 9.3.4]{Bogachev:2007aa}). In this setting, it was shown in \cite{Halmos:1944wm} (see also \cite[`Category']{Halmos:1960vh}) that among invertible measure-preserving transformations of $(X,\Sigma,\mu)$, the weakly mixing ones are generic. This result was later adapted to topological dynamical systems.

\begin{theorem}[\cite{Katok:1970wx}] \label{thm:mixingman}
Let $X$ be a compact, connected topological manifold and let $\mu$ be an OU measure on $X$. Then, the weakly mixing elements of $\mathcal{H}(X,\mu)$ form a dense $G_\delta$ set.
\end{theorem}

The conjugacy lemma of \cite{Alpern:1979th} (see also \cite[Theorem 10.1]{Alpern:2000ux}), which is a strengthening of the prototypical result of Halmos, established a very general framework for transferring generic properties of measure-preserving systems to topological ones. Namely (see \cite[Theorem 10.3]{Alpern:2000ux}), if $\mathcal{V}$ is a $G_\delta$ subset of invertible measure-preserving transformations (with respect to the weak topology) that is invariant under conjugation and contains an antiperiodic transformation, then $\mathcal{V}\cap\mathcal{H}(X,\mu)$ is a dense $G_\delta$ subset of $\mathcal{H}(X,\mu)$ (with respect to the uniform topology). Theorem~\ref{thm:mixingman} is a special case of this observation. In \S~\ref{section:classification}, we will use classification to lift from topological systems (viewed as actions on trace spaces) to $\cs$-algebras.

Various genericity results for noninvertible transformations of $([0,1],\mu)$, where $\mu$ is an OU measure on the interval $[0,1]$, are established in \cite{Bobok:2020uj}. In particular, the following holds.

\begin{theorem}[\cite{Bobok:2020uj}] \label{thm:genint}
In the set $\mathcal{C}([0,1],\mu)$ of $\mu$-preserving continuous maps $[0,1]\to [0,1]$, both the weakly mixing and (maximally) chaotic elements form dense $G_\delta$ sets. The strongly mixing elements are dense but meagre.
\end{theorem}

\begin{remark} \label{remark:flipside}
Rather than the generic behaviour of $h\in\mathcal{H}(X,\mu)$ (or $h\in\mathcal{C}(X,\mu)$) for a given $\mu$, one might ask for the generic behaviour of $\mu\in\mathcal{M}(X)^h$, that is, of an $h$-invariant Borel probability measure $\mu$, for a given homeomorphism (or continuous map) $h$. It is shown in \cite{Catsigeras:2019tq} that for a generic continuous $h$ on a $C^1$ compact, connected manifold, ergodic measures are nowhere dense in $\mathcal{M}(X)^h$. On the other hand, for an irreducible diffeomorphism on a compact manifold without boundary, ergodicity is generic \cite[Theorem 7.1]{Gelfert:2018we}.  For a mixing axiom $A$ diffeomorphism, the generic invariant measure is OU \cite{Sigmund:1970wa} and weakly mixing \cite{Sigmund:1972vs}, and the strongly mixing invariant measures are dense \cite{Sigmund:1972vs} but meagre \cite{Sigmund:1970wa}.
\end{remark}

\section{Lifting via classification} \label{section:classification}

The standing assumption in this section is that $A$ is an infinite-dimensional, separable, algebraically simple $\cs$-algebra whose tracial state space $T(A)$ is nonempty and compact, with compact extreme boundary $\partial_e(T(A))$. (In conjunction with $\mathcal{Z}$-stability, these conditions in particular imply that $A$ has `continuous scale'; see, for example, \cite[\S5]{Elliott:2020vp}.) Throughout, the boundary $\partial_e(T(A))$ is denoted by $X_A$ ($X$ for `extreme'). We equip $T(A)$ with the $w^*$-topology, and fix a metric $d$ on $X_A$ that induces the subspace topology. The set of continuous affine maps $T(A)\to\rr$ is denoted by $\aff(T(A))$.

\subsubsection*{Endomorphism spaces}
Let us write $\End(A,X_A)$ for the set of what one might call `tracially nondegenerate' endomorphisms of $A$, that is, $^*$-homomorphisms $\alpha\colon A\to A$ that induce continuous affine maps $\alpha^*=T(\alpha)\colon T(A)\to T(A)$ that preserve $X_A=\partial_e(T(A))$. (Note, for example, that the shift endomorphisms of $M_{n^\infty}$ discussed in \S~\ref{subsection:rokhlin} are not covered by this definition.) For a given $\tau\in T(A)$, let us write $\End(A,X_A,\tau)$ for those elements of $\End(A,X_A)$ that fix $\tau$. Note that
\[
\Aut(A)\subseteq\bigcup_{\tau\in T(A)}\End(A,X_A,\tau)\subseteq\End(A,X_A)
\]
(that is, every automorphism preserves $X_A$ and fixes some $\tau$). We write $\Aut(A,\tau)$ for the set $\Aut(A)\cap\End(A,X_A,\tau)$.

Each of these spaces is equipped with a suitable topology of pointwise convergence. For  $\End(A,X_A)$ and $\End(A,X_A,\tau)$ this means the topology induced by the family of pseudometrics
\[
\{d_F(\alpha,\beta)=\max_{a\in F}\|\alpha(a)-\beta(a)\| \mid F\subseteq A \text{ finite}\}.
\]
To ensure that $\Aut(A)$ and $\Aut(A,\tau)$ are Polish spaces, we should use the finer topology provided by
\begin{equation} \label{eqn:polish2}
\{d_F(\alpha,\beta)+d'_F(\alpha^{-1},\beta^{-1}) \mid F\subseteq A \text{ finite}\},
\end{equation}
where $d'_F(\varphi,\psi) = \inf_{u}\max_{a\in F}\|u\varphi(a)u^*-\psi(a)\|$, the infimum taken over all unitaries in (the minimal unitisation of) $A$. (Completeness follows from an approximate intertwining; see, for example, \cite[Corollary 2.3.3]{Rordam:2002yu}.) That said, we will not apply the Baire category theorem directly in $\Aut(A)$, instead lifting generic properties from the trace space to the $\cs$-algebra via classification.

\subsubsection*{Representing measures} \label{ssc:rep} Since $T(A)$ is a Choquet simplex, every $\tau\in T(A)$ is represented by a unique Borel probability measure $\mu=\mu_\tau$ supported on $X_A=\partial_e(T(A))$; that is, $\mu$ is the unique measure on $X_A$ such that
\begin{equation} \label{eqn:rep1}
f(\tau) = \int_{X_A} f\,d\mu  \quad\text{ for every } f\in\aff(T(A)). 
\end{equation}
(The metric $d$ on $X_A$ can then be extended to all of $T(A)$ via a choice of Wasserstein metric between representing measures; see \cite[\S2]{Jacelon:2021vc}).

In fact, every continuous affine functional $f\colon T(A)\to\mathbb{R}$ is of the form $f=\hat a$ for some self-adjoint element $a\in A$ (see \cite[Theorem 9.3]{Lin:2007qf}, whose use of \cite[\S2]{Cuntz:1979fv} does not 
even require simplicity). The defining property (\ref{eqn:rep1}) of the representing measure $\mu$ is therefore
\begin{equation} \label{eqn:rep2}
\tau(a) = \int_{X_A} \hat a\,d\mu \quad\text{ for every } a\in A_{sa}.
\end{equation}

\sloppy
Conversely, every measure $\mu$ defines via equation (\ref{eqn:rep2}) a trace $\tau=\tau_\mu\in T(A)$. If $\mu$ is an OU measure, we call $\tau_\mu$ an \emph{OU trace}.

We use representing measures to reduce the analysis of generic properties of the space $\aff_{X_A}(T(A),T(A))$ of continuous $X_A$-preserving affine maps $T(A)\to T(A)$ to those of $C(X_A,X_A)$. In particular, we can extend $h\in C(X_A,X_A)$ to an element of $\aff_{X_A}(T(A),T(A))$ via the pushforward
\begin{equation} \label{eqn:push}
h_*(\tau)(a)=\int_{X_A}\hat a\circ h\,d\mu_\tau \quad\text{ for every } a\in A.
\end{equation}

\begin{lemma} \label{lemma:extend}
The pushforward extension $h\mapsto h_*$ of equation (\ref{eqn:push}) gives a homeomorphism between $C(X_A,X_A)$ and $\aff_{X_A}(T(A),T(A))$ (with respect to the $d$-uniform topology).
\end{lemma}

\begin{proof}
The map $h\mapsto h_*$ is continuous, and its continuous inverse is the restriction map $\aff_{X_A}(T(A),T(A)) \to C(X_A,X_A)$.
\end{proof}
\fussy

The map $T(\cdot)\colon\End(A,X_A)\to C(X_A,X_A)$, $\alpha\mapsto\alpha^*|_{X_A}$, is continuous. It sends $\End(A,X_A,\tau)$ to the set $\mathcal{C}(X_A,\mu)$ of $\mu=\mu_\tau$-preserving continuous maps $X_A\to X_A$, and $\Aut(A,\tau)$ to $\mathcal{H}(X_A,\mu)$ (continuously with respect to equations (\ref{eqn:polish1}) and (\ref{eqn:polish2})). The following is immediate.

\begin{lemma} \label{lemma:gdlift}
If $\mathcal{V}$ is an open (respectively, $G_\delta$) subset of $\mathcal{C}(X_A,\mu)$, then $T^{-1}\mathcal{V}$ is an open (respectively, $G_\delta$) subset of $\End(A,X_A,\tau_\mu)$ that is invariant under approximate unitary equivalence. The same is true of $T^{-1}\mathcal{V}\cap\Aut(A)\subseteq\Aut(A,\tau_\mu)$ for $\mathcal{V}\subseteq\mathcal{H}(X_A,\mu)$.
\end{lemma}

If we further demand that $A$ be classifiable by the Elliott invariant and impose suitable restrictions on $K$-theory and traces, we can also lift dense sets to dense sets.

First we recall some notation. Every $\tau\in T(A)$ extends to a tracial state on the minimal unitisation $\tilde A$ of $A$ via $a+\lambda1 \mapsto \tau(a)+\lambda$ and also to a trace $\tau\otimes\tr_k$ on any matrix algebra $A\otimes M_k \cong M_k(A)$ over $A$. We denote these extensions also by $\tau$. For a unital $\cs$-algebra $A$, the pairing map $\rho_A\colon K_0(A)\to\aff(T(A))$ is the homomorphism defined by $\rho_A([p])(\tau)=\tau(p)$ for $p\in M_k(A)$ a projection.

If $A$ is nonunital, then $K_0(A)$ is the kernel of the map $K_0(\tilde A)\to K_0(\cc)$ induced by the quotient map $\Pi_A\colon \tilde A \to \cc$, and the pairing is defined by $\rho_A([p]-[q])(\tau)=\tau(p)-\tau(q)$. Moreover, $T(\tilde{A})$ is affinely homeomorphic to the convex hull of $T(A)$ and the trace $\tau_\cc$ induced by $\Pi_A$; in particular, every continuous affine map from $T(A)$ to itself extends uniquely to one from $T(\tilde{A})$ to itself that fixes $\tau_\cc$.

\begin{definition} \label{def:trivial}
A $\cs$-algebra $A$ is said to have \emph{trivial tracial pairing} if either
\begin{enumerate}[1.]
\item $A$ is unital and the image of $\rho_A\colon K_0(A)\to\aff(T(A))$ is contained in the constant functions, or
\item $A$ is nonunital and $\ker\rho_A=K_0(A)$.
\end{enumerate}
\end{definition}

Note that a simple $\cs$-algebra $A$ must in fact be stably projectionless if $\ker\rho_A=K_0(A)$. This condition holds, for example, for the $\cs$-algebras classified in \cite{Gong:2020vg}. Examples of unital $\cs$-algebras with trivial tracial pairing are limits of subhomogeneous building blocks with connected spectra (in particular, the interval algebras $C([0,1],M_n)$ considered in \cite{Thomsen:1994qy} and the prime dimension drop algebras considered in \cite{Jiang:1999hb}).

\begin{theorem} \label{lemma:denselift}
Suppose that, in addition to the standing assumption of this section, $A$ also satisfies the UCT and has finite nuclear dimension, trivial tracial pairing and torsion-free $K_1$. Then, for every dense subset $\mathcal{V}$ of $\mathcal{C}(X_A,\mu)$, $T^{-1}\mathcal{V}$ is a dense subset of $\End(A,X_A,\tau_\mu)$. The same is true of $T^{-1}\mathcal{V}\cap\Aut(A)\subseteq\Aut(A,\tau_\mu)$ for $\mathcal{V}\subseteq\mathcal{H}(X_A,\mu)$.
\end{theorem}

\begin{proof}
We will prove the statement for automorphisms, noting that the same argument works for endomorphisms. Let $\alpha\in\Aut(A)$ with $T(\alpha)\in\mathcal{H}(X_A,\mu)$ (that is, $\tau_\mu\circ\alpha=\tau_\mu$). Let $F\subseteq A$ be finite and let $\varepsilon>0$. We must show that there is an automorphism $\beta$ of $A$ and a unitary $w$ such that $T(\beta)\in\mathcal{V}$,
\begin{equation} \label{eqn:pnorm1}
\max_{a\in F}\|\beta(a)-\alpha(a)\| < \varepsilon
\end{equation}
and
\begin{equation} \label{eqn:pnorm2}
\max_{a\in F}\|\beta^{-1}(a)-w\alpha^{-1}(a)w^*\| < \varepsilon.
\end{equation}

To do this, we use the fact that under our hypotheses (see \cite{Gong:2021va,Gong:2020vg,Carrion:wz}) there is an invariant $\Inv$ based on $K$-theory and traces that classifies morphisms (that is, tracially nondegenerate $^*$-homomorphisms) $A\to A$, meaning:
\begin{description}
\item[(existence)] every morphism $\Inv(A)\to\Inv(A)$ can be lifted to a morphism $A\to A$;
\item[(uniqueness)] for every morphism $\varphi\colon A\to A$ and $F$, $\varepsilon$ as above, there is $\delta>0$ such that, if $\psi\colon A\to A$ is a morphism for which $\Inv(\psi)$ $\delta$-agrees with $\Inv(\varphi)$ (made precise below), then there is a unitary $u$ in the minimal unitisation $\tilde{A}$ of $A$ with $\max_{a\in F}\|u\psi(a)u^*-\varphi(a)\| < \varepsilon$.
\end{description}
Given such an invariant, the strategy is to:
\begin{enumerate}[1.]
\item perturb $\Inv(\alpha)$ by keeping its $K$-theory part the same but replacing its tracial part $T(\alpha)$ by a nearby $h\in\mathcal{V}$;
\item use \textbf{(existence)} to lift this perturbed $\Inv$-morphism to a morphism $\alpha_h\colon A\to A$;
\item use \textbf{(uniqueness)} to deduce equations (\ref{eqn:pnorm1}) and (\ref{eqn:pnorm2}) for a unitary conjugate $\beta$ of $\alpha_h$.
\end{enumerate}
For this to work, we must know that the perturbation in Step 1 still gives a valid $\Inv$-morphism. This  is where assuming something like trivial tracial pairing becomes essential and where we must pay close attention to the actual structure of $\Inv$. Finally, then, here are its components:
\begin{description}
\item [Traces] Like the Elliott invariant, $\Inv$ includes the trace functor $T(\cdot)$ (or, dually, $\aff(T(\cdot))$).
\item[Total $K$-theory] $\Inv$ includes not just the usual $K$-groups $K_*(A)$, but $K$-theory with coefficients $\underline{K}(A) = \bigoplus_{n=0}^\infty K_*(A;\zz/n)$. By the universal multicoefficient theorem for separable $\cs$-algebras satisfying the UCT (see \cite[\S 1.4]{Dadarlat:1996aa}), the group of homomorphisms $\underline{K}(A) \to \underline{K}(A)$ that respect the `Bockstein operations' is isomorphic to the group $KL(A,A)$ (see \cite[2.4.8]{Rordam:2002yu}). For $\kappa\in KL(A,A)$ and $i=0,1$, we write $\kappa_i$ for associated homomorphism $K_i(A)\to K_i(A)$.
\item[Hausdorffised algebraic $K_1$]  Let $CU(\tilde{A})$ denote the closure of the commutator subgroup of the unitary group $U(\tilde{A})$ of $\tilde{A}$. Write $\overline{K_1}^{alg}(A):=U(\tilde{A})/CU(\tilde{A})$. The $\cs$-algebras we consider here are all of stable rank one (see \cite[Theorem 6.7]{Rordam:2004kq} and \cite[Corollary 6.8]{Fu:2022aa}), so we are justified in making this definition without passing to matrix algebras over $A$. Moreover, as in \cite[Lemma 3.1]{Nielsen:1996aa} we have a short exact sequence
\begin{equation} \label{eqn:thomsen}
\begin{tikzcd}
0 \arrow[r] & \aff(T(\tilde{A}))/\overline{\rho_A(K_0(\tilde{A}))} \arrow[r,"\lambda_A"] & \overline{K_1}^{alg}(A) \arrow[r,"\pi_A"] & K_1(A) \arrow[r] & 0.
\end{tikzcd}
\end{equation}
Here, $\pi_A$ is the canonical surjection $\pi_A([u]_{alg})=[u]_1$. By divisibility of the group $\aff(T(\tilde{A}))/\overline{\rho_A(K_0(\tilde{A}))}$, the sequence (\ref{eqn:thomsen}) splits (unnaturally). We fix a splitting map $s_A\colon K_1(A)\to \overline{K_1}^{alg}(A)$ that is a right inverse of $\pi_A$.

The inclusion $\lambda_A$ is the inverse of the map $U_0(\tilde{A})/CU(\tilde{A}) \to \aff(T(\tilde{A}))/\overline{\rho_A(K_0(\tilde{A}))}$ induced by the de la Harpe--Skandalis determinant (see \cite[\S3]{Thomsen:1995qf}, and note that stable rank one ensures that $CU(\tilde{A})$ is contained in the connected component $U_0(\tilde{A})$ of $U(\tilde{A})$). It is an isometry with respect to the quotient metric on $U_0(\tilde{A})/CU(\tilde{A})$ and the metric $d_A$ on $\aff(T(\tilde{A}))/\overline{\rho_A(K_0(\tilde{A}))}$ obtained by adjusting the quotient metric $d'_A$ to
\begin{equation} \label{eqn:metric}
d_A(f,g) =
\begin{cases}
\left|e^{2\pi i d'_A(f,g)}-1\right| & \text{if } d'_A(f,g) < \frac{1}{2}\\
2 & \text{if } d'_A(f,g) \ge \frac{1}{2}.
\end{cases}
\end{equation}
\end{description}

An $\Inv$-morphism consists of: an element $\kappa\in KL(A,A)$ (which if $A$ is unital is required to satisfy $\kappa_0([1_A]_0)=[1_A]_0$); a continuous affine map $\kappa_T\colon T(A)\to T(A)$ (inducing a positive unital linear map $\kappa_T^*\colon\aff(T(\tilde{A}))\to\aff(T(\tilde{A}))$) that is compatible with $\kappa$ in the sense that $\rho_A\circ\kappa_0 = \kappa_T^*\circ\rho_A$ (therefore inducing an endomorphism $\overline{\kappa_T}$ of $\aff(T(\tilde{A}))/\overline{\rho_A(K_0(\tilde{A}))}$); and a homomorphism $\kappa_U\colon \overline{K_1}^{alg}(A) \to \overline{K_1}^{alg}(A)$ that is compatible with $\overline{\kappa_T}$ in the sense that the diagram
\begin{equation} \label{eqn:compatible}
\begin{tikzcd}
K_0(\tilde{A}) \arrow[r,"\rho_A"] \arrow[d,"\kappa_0"] & \aff(T(\tilde{A})) \arrow[r,"\lambda_A"] \arrow[d,"\kappa_T^*"] & \overline{K_1}^{alg}(A) \arrow[r,"\pi_A"] \arrow[d,"\kappa_U"] & K_1(A) \arrow[d,"\kappa_1"]\\
K_0(\tilde{A}) \arrow[r,"\rho_A"] & \aff(T(\tilde{A})) \arrow[r,"\lambda_A"] & \overline{K_1}^{alg}(A) \arrow[r,"\pi_A"] & K_1(A)
\end{tikzcd}
\end{equation}
commutes. (Importantly, the necessity of another family of commutative diagrams between $K_0$ with coefficients and $\overline{K_1}^{alg}(A)$ is identified in \cite{Carrion:wz}. There it is also shown that this additional compatibility is automatic if $K_1(A)$ is torsion free.)

That $\Inv$ satisfies \textbf{(existence)} is provided by \cite[Corollary 5.13]{Gong:2021va} (see also \cite{Carrion:wz}), or in the stably projectionless case, \cite[Theorem 12.8]{Gong:2020vg} (noting in this latter case that $A\in\mathcal{D}_0$ by \cite[Theorem 15.2]{Gong:2020vg} and the proof of \cite[Theorem 15.5]{Gong:2020vg}, and that by \cite[Theorem 13.1]{Gong:2020vg}, the models $B_T$ constructed in \cite[\S7]{Gong:2020vg} cover the full class $\mathcal{D}_0$; in other words, the hypotheses of \cite[Theorem 12.8]{Gong:2020vg} hold for $B_T=A$).

We use \textbf{(existence)} as follows: For any $h\in C(X_A,X_A)$ (which we extend by Lemma~\ref{lemma:extend} to $\aff_{X_A}(T(A),T(A))$), there exists $\alpha_h\in\End(A,X_A)$ (which by \cite[Theorem 29.5]{Gong:2020uf}, or in the stably projectionless case, \cite[Theorem 13.1]{Gong:2020vg}, is an automorphism if $h$ is invertible) with $\kappa_T:=T(\alpha_h)=h$, $\kappa:=KL(\alpha_h)=KL(\alpha)$ and $\kappa_U:=\overline{K_1}^{alg}(\alpha_h)$ defined by $\kappa_U\circ\lambda_A = \lambda_A\circ\kappa_T^*$ and $\kappa_U\circ s_A = \overline{K_1}^{alg}(\alpha)\circ s_A \colon K_1(A) \to \overline{K_1}^{alg}(A)$. Note that these choices do provide a valid $\Inv$-morphism: The middle square of equation (\ref{eqn:compatible}) commutes by construction, the right square commutes because it does for $\Inv(\alpha)$, and the left square commutes because $A$ has trivial tracial pairing.

For \textbf{(uniqueness)}, we appeal to the approximate version of \cite[Theorem 4.3]{Gong:2021va} (see also \cite{Carrion:wz}, which includes a classification of $^*$-homomorphisms into sequence algebras), or in the stably projectionless case, \cite[Theorem 5.3]{Gong:2020vg} (again noting that our assumptions on $A$ ensure that the hypotheses of this theorem are indeed satisfied for $B=A$: by \cite[Remark 3.11]{Gong:2020vg}, $A\in\mathcal{D}_0\subseteq\mathcal{D}^d$; the map $\mathbf{T}\colon \nn\times\nn\to\nn$ can be taken to be $(n,k)\mapsto n$ (cf.\ \cite[5.2]{Gong:2020vg}); and the fullness condition can be dropped since we work only with tracially nondegenerate genuine $^*$-homomorphisms rather than more general approximately multiplicative maps (cf.\ \cite[Remark 5.6]{Gong:2020vg})).

By commutativity of equation (\ref{eqn:compatible}) and the fact that (the inverse of) $\lambda_A$ is an isometry (with respect to the quotient metric on $\overline{K_1}^{alg}(A)$ and the metric $d_A$ of equation (\ref{eqn:metric})) we can phrase \textbf{(uniqueness)} as follows: For our fixed finite set $F\subseteq A$ and tolerance $\varepsilon>0$, there is $\delta>0$ such that, if
\begin{equation} \label{eqn:h1}
\sup_{x\in X_A}(h(x),T(\alpha)(x))<\delta,
\end{equation}
then there is a unitary $u$ such that equation (\ref{eqn:pnorm1}) holds for $\beta:=u\alpha_h(\cdot)u^*$, and if moreover
\begin{equation} \label{eqn:h2}
\sup_{x\in X_A}(h^{-1}(x),T(\alpha)^{-1}(x))<\delta,
\end{equation}
then equation (\ref{eqn:pnorm2}) also holds for some unitary $w$. By density of $\mathcal{V}$, there exists $h\in\mathcal{V}$ such that equations (\ref{eqn:h1}) and (\ref{eqn:h2}) hold, so we are done.
\end{proof}

Combining Lemma~\ref{lemma:gdlift}, Theorem~\ref{lemma:denselift}, Theorem~\ref{thm:chaoticman}, Theorem~\ref{thm:mixingman} and Theorem~\ref{thm:genint}, we immediately have the following.

\begin{theorem} \label{thm:tracialgenerics}
Let $A$ be an infinite-dimensional, separable, algebraically simple $\cs$-algebra that satisfies the UCT, has finite nuclear dimension, trivial tracial pairing and torsion-free $K_1$, and whose tracial state space $T(A)$ is nonempty and compact, with compact extreme boundary $X_A=\partial_e(T(A))$ that has the structure of a topological manifold of dimension $n \ge 2$. Let $\mu$ be an OU measure on $X_A$, and let $\tau=\tau_\mu$ be the corresponding element of $T(A)$. Then,
\[
\{\alpha\in\Aut(A,\tau) \mid T(\alpha)\in\mathcal{H}(X_A,\mu) \text{ is weakly mixing and chaotic}\}
\]
is a residual subset of $\Aut(A,\tau)$ (that is, contains a dense $G_\delta$ set). If $X_A=[0,1]$, then
\[
\{\alpha\in\End(A,X_A,\tau) \mid T(\alpha)\in\mathcal{C}(X_A,\mu) \text{ is weakly mixing and chaotic}\}
\]
is a residual subset of $\End(A,X_A,\tau)$, and
\[
\{\alpha\in\End(A,X_A,\tau) \mid T(\alpha)\in\mathcal{C}(X_A,\mu) \text{ is strongly mixing}\}
\]
is a dense but meagre subset of $\End(A,X_A,\tau)$.
\end{theorem}

\begin{remark}
\begin{enumerate}[1.]
\item Torsion-free $K_1$ can be replaced by one of the other conditions in \cite[Corollary 5.13]{Gong:2021va} (for example, tracial rank $\le 1$) and in fact by results of the forthcoming work \cite{Carrion:wz} can be removed as an assumption altogether. (Briefly, the map $\overline{K_1}^{alg}(\alpha_h)$ constructed in the proof of Theorem~\ref{lemma:denselift} can be adjusted without changing $T(\alpha_h)$ so that equation (\ref{eqn:compatible}) still commutes, and so do the additional compatibility diagrams mentioned in the proof.) On the other hand, triviality of the tracial pairing is used rather crucially.
\item Suppose that $A$ is unital with $X_A=\partial_e(T(A))$ nonempty. Then, $\partial_e(T(\bigotimes_{n\in\zz}A)) \cong \prod_{n\in\zz}X_A$ (see \cite[Proposition 3.5]{Bosa:aa}). The (right) shift automorphism of $\bigotimes_{n\in\zz}A$ induces the (left) shift on $\prod_{n\in\zz}X_A$, which is strongly mixing with respect to $\bigotimes_{n\in\zz}\mu$, for any Borel probability measure $\mu$ on $X_A$. This demonstrates the existence of automorphisms with strongly mixing actions on trace spaces for a collection of $\cs$-algebras $B$ with $\partial_e(T(B))$ a Cantor space or Hilbert cube.
\item In \cite[\S4]{Thomsen:1994qy}, Thomsen shows how to construct an approximately interval (AI) algebra that admits tracially chaotic endomorphisms. By Theorem~\ref{thm:tracialgenerics}, these endomorphisms are in fact typical for this algebra. 
\end{enumerate}
\end{remark}

\section{Strongly chaotic dynamics} \label{section:strongchaos}

In this section, we describe some well-known examples of ergodic topological dynamical systems $(X,\mu,h)$ on compact metric spaces with striking statistical properties. We then construct model $\cs$-algebras admitting endomorphisms that exhibit these phenomena.

\subsection{Statistical features} \label{subsection:statistics}

It is straightforward to check that strong mixing is equivalent to
\begin{equation} \label{eqn:mix}
\lim_{n\to\infty} \int_X (f\circ h^n)g\;d\mu = \int_Xf\,d\mu \cdot \int_Xg\,d\mu
\end{equation}
for every $f,g\in L^2(X,\mu)$. It is sufficient to consider only observables that are continuous and in fact sufficient to verify equation (\ref{eqn:mix}) for observables $f=g$.

We will write $\hol_\eta(X)$ for the set of \emph{H\"{o}lder continuous} functions $f\colon X\to\mathbb{R}$ with exponent $\eta>0$, that is,
\[
\hol_\eta(X) = \{f\colon X\to\mathbb{R} \mid \exists C>0\:\forall x,y\in X\: (|f(x)-f(y)| \le Cd(x,y)^\eta)\}.
\]

\begin{definition} \label{def:edc}
The system $(X,\mu,h)$ has \emph{exponential decay of correlations (EDC)} (for H\"{o}lder continuous observables) if for every $\eta>0$ there exists $\gamma\in(0,1)$ such that for every $f,g\in\hol_\eta(X)$ there exists $C>0$ such that for every $n\in\nn$,
\[
\left|\int_X (f\circ h^n)g\;d\mu - \int_Xf\,d\mu \cdot \int_Xg\,d\mu\right| \le C\gamma^n.
\]
\end{definition}

The central limit theorem (CLT) is closely related to exponential mixing. For instance, for the dispersing billiards described in \S~\ref{subsection:examples} below, the CLT can be deduced from correlation bounds (see \cite{Chernov:2006tb}). Given an observable $f\colon X\to\mathbb{R}$, we write $S_nf$ for the ergodic sum $\sum_{k=0}^{n-1}f\circ h^k$.

\begin{definition} \label{def:clt}
The system $(X,\mu,h)$ satisfies the \emph{CLT} (for H\"{o}lder continuous observables) if for every $\eta>0$ and every $f\in\hol_\eta(X)$ that is not a \emph{coboundary} and satisfies $\int_Xf\,d\mu=0$, there exists $\sigma_f > 0$ such that the sequence of random variables $\frac{1}{\sqrt{n}}S_nf$ converges in distribution as $n\to\infty$ to the normal distribution $\mathcal{N}(0,\sigma_f^2)$. In other words, for every $z\in\rr$,
\begin{equation} \label{eqn:normal}
\lim_{n\to\infty} \mu\left(\left\{x\in X \mid \frac{S_nf(x)}{\sqrt{n}} \le z \right\}\right) = \frac{1}{\sigma_f\sqrt{2\pi}}\int_{-\infty}^z\exp\left(-\frac{t^2}{2\sigma_f^2}\right)dt.
\end{equation}
\end{definition}
Here, the variance $\sigma_f^2$ can be computed as $\sigma_f^2 = \lim_{n\to\infty} \frac{1}{n} \int_X (S_nf)^2d\mu$. For coboundaries, that is, functions $f$ for which $f=g\circ h-g$ for some $g\in L^2(X,\mu)$, the variance is $0$ and equation (\ref{eqn:normal}) holds provided the right-hand side is interpreted as the Heaviside function (that is, $\frac{S_nf}{\sqrt{n}}$ converges almost surely to $0$).

The CLT provides considerably more information than Birkhoff's ergodic theorem, which for ergodic $h$ says that, for every integrable $f$ and almost every $x$, $\lim_{n\to\infty}\frac{1}{n}S_nf(x)=\int_X f\,d\mu$. That said,  in the compact, uniquely ergodic setting, the following finite-time estimate of large deviation from the mean is available even if the CLT is not.

\begin{proposition} [\cite{Benoist:2016ur}] \label{prop:breiman}
Suppose that $X$ is a compact metric space and $(X,\mu,h)$ is uniquely ergodic (that is, $\mu$ is the unique invariant measure of $h$). Then, for every $\varepsilon>0$ and every $k\in\nn$, there exist constants $c_1,c_2>0$ such that, for every $k$-Lipschitz $f\colon X\to\rr$ and every $n\in\nn$,
\[
\mu\left(\left\{x\in X \mid \left|\frac{1}{n}S_nf(x)-\int_X f\,d\mu\right|>\varepsilon\right\}\right) \le c_1e^{-c_2n\varepsilon^2}.
\]
\end{proposition}
Actually, Proposition~\ref{prop:breiman} holds for arbitrary continuous observables, not just Lipschitz ones, but in general the constants $c_1$ and $c_2$ will depend on $f$ (see \cite[Proposition 3.1]{Benoist:2016ur}). On the other hand, for systems like axiom A diffeomorphisms (see \cite[\S5]{Chazottes:2015ti}) and holomorphic endomorphisms of projective space (with $n\varepsilon^2$ replaced by $n(\log n)^{-2}p(\varepsilon)$ for a suitable function $p$; see \cite{Dinh:2010tg}), the constants do not depend on $\varepsilon$.

\subsection{Examples} \label{subsection:examples}

For the convenience of those not overly familiar with the statistics of dynamical systems, we present several well-known examples where the features described in \S~\ref{subsection:statistics} can be observed. The reader should bear these examples in mind as candidate tracial dynamical systems to which Theorem~\ref{thm:models} and Theorem~\ref{thm:lift} can be applied (provided that the space $X$ is connected and the map $h\colon X\to X$ is continuous).

\subsubsection*{Subshifts of finite type} Suppose that $A$ is a $0$-$1$ matrix which is mixing (that is, for some $M\in\nn$, all of the entries of $A^M$ are nonzero), $h\colon\Sigma_A\to\Sigma_A$ is the associated subshift of finite type and $\mu=\mu_\varphi$ is the Gibbs measure associated to some potential $\varphi$. Then, the CLT and EDC hold \cite[\S1E]{Bowen:2008uh}. Using Markov partitions \cite[\S3C]{Bowen:2008uh} to construct symbolic dynamics \cite[\S3D]{Bowen:2008uh}, one sees that they also hold for mixing axiom A diffeomorphisms \cite[Theorem 4.1]{Bowen:2008uh}.

\subsubsection*{Expanding circle maps} These types of dynamical systems are used to model the `intermittency' of turbulent flows, which transition between periodic and chaotic behaviour (see  \cite{Pomeau:1980vl}). Here is one example (see \cite{Liverani:1999ug} and \cite[\S3.5]{Chazottes:2015ti}). Let $\alpha\in\left(0,1/2\right)$ and define $h\colon[0,1]\to[0,1]$ by
\[
 h(t)=
 \begin{cases}
 t+2^\alpha t^{1+\alpha} & \text{ if }\ 0\le t < \frac{1}{2}\\
 2t-1 & \text{ if }\ \frac{1}{2} \le t\le 1.
 \end{cases}
\]
Identifying the boundary points of $[0,1]$, we view this as a continuous map $S^1\to S^1$. There is a unique ergodic invariant probability measure $\mu$ that is equivalent to Lebesgue measure. The dynamical system $(S^1,\mu,h)$ has polynomial (rather than exponential) mixing rates but still satisfies the CLT \cite[Theorem 6]{Young:1999vu}. Moreover, the \emph{almost-sure CLT} holds for any Lipschitz observable $f\colon S^1\to\mathbb{R}$ (see \cite[Theorem 18]{Chazottes:2015ti}): If $\int fd\mu=0$ and $f$ is not a coboundary (so that the variance $\sigma_f^2$ of $f$ is nonzero), then for $\mu$-a.e.\ $t\in S^1$, the sequence of weighted averages $\frac{1}{D_n}\sum_{k=1}^n\frac{1}{k}\delta_{S_kf(t)/\sqrt{k}}$, where $\delta_x$ denotes the point mass at $x$ and $D_n=\sum_{k=1}^n\frac{1}{k}$, is $w^*$-convergent to $\mathcal{N}(0,\sigma_f^2)$.

For $\alpha=1/2$, the normalising factor $\sqrt{n}$ in equation (\ref{eqn:normal}) must be replaced by $\sqrt{n\log n}$, and for parameters $\alpha\in(1/2,1)$, there is convergence to a suitable non-Gaussian `stable law'; see \cite[Theorem 9]{Chazottes:2015ti}.

\subsubsection*{Billiards} A \emph{billiard table} is a closed connected domain $Q$ in the plane $\rr^2$ or the torus $\tot^2$ whose boundary $\partial Q$ consists of finitely many simple $C^3$ curves that meet each other only at their ends. The \emph{billiard} on $Q$ is the dynamical system in the unit tangent bundle $M$ of the ambient manifold restricted to $Q$ generated by the unit-speed motion of a tangent vector along a geodesic, whose interaction with $\partial Q$ is governed by the `angle of incidence equals angle of reflection' rule. Billiards are analysed via the discrete dynamical system $(X,h,\mu)$, where $X=\partial Q\times[-\frac{\pi}{2},\frac{\pi}{2}]$ (parameterised by $(r,\varphi)$, where $r$ is arc length along $\partial Q$ and $\varphi$ is the angle of incidence relative to an inward-facing normal), $h\colon X\to X$ is the collision map and $\mu$ is normalised Liouville measure $\mu=c\cos\varphi dr d\varphi$. The behaviour of the system depends on the shape of the boundary. \emph{Dispersing billiards} (for example, the periodic Lorentz gas), that is, those with convex boundary curves, admit the CLT and EDC (see \cite{Bunimovich:1991we, Chernov:1999wk, Chernov:2006tb}). Certain billiards called \emph{Bunimovich stadia} \cite{Bunimovich:1979uq} whose boundaries consist only of \emph{focusing} (concave) and \emph{neutral} (rectilinear) components satisfy the CLT \cite{Balint:2006ve} (in some cases requiring $\sqrt{n}$ to be replaced by $\sqrt{n\log n}$) but have polynomial mixing rates \cite{Chernov:2005up}. In both cases, the almost-sure CLT is also observed to hold (see, for example, \cite{Chazottes:2005tp} or \cite{Leppanen:2017tv} for dispersing billiards and \cite{Chazottes:2007wg} for stadia).

Unfortunately, the map $h$ will always have discontinuities (where boundary curves intersect nonsmoothly but also at instances of `grazing' collisions), so it does not seem that billiards are immediately suitable for $\cs$-dynamics. However, they can be modelled via certain countable Markov partitions called `Young towers' \cite{Young:1998uq}, so while not covered by the constructions of \S~\ref{subsection:models}, they could be  represented as actions on zero-dimensional spaces.

\subsubsection*{Complex dynamics} The CLT and EDC are exhibited by many complex-geometric dynamical systems $(X,\mu,h)$, with $\mu$ the measure of maximal entropy, such as: holomorphic endomorphisms of projective space \cite{Dupont:2010tt, Dinh:2010tg}, holomorphic automorphisms of positive entropy on compact K\"{a}hler surfaces \cite{Dinh:2010wo} (at least EDC) and meromorphic automorphisms of the Riemann sphere \cite{Denker:1996vk} (with $X$ the Julia set of $h$; note that in some cases, the Julia set is known to be connected \cite{Inninger:2004wn,Peherstorfer:2001tq} or even the whole Riemann sphere \cite{Inninger:2002wj}).

\subsection{Model building} \label{subsection:models}

Finally in this section, we show how to construct $\cs$-models that witness prescribed topological dynamical systems at the level of the trace space.

\begin{theorem} \label{thm:models}
Let $(X,d)$ be a compact, connected Riemannian manifold. Then, there exists a separable, simple, unital, nuclear, $\js$-stable, projectionless $\cs$-algebra $A$ that has trivial tracial pairing and
satisfies the UCT such that $\partial_e(T(A))\cong X$ and $\{a \mid \hat a \in \lip(X,d)\}$ is dense in $A_{sa}$. 
\end{theorem}

\begin{proof}
The construction is based on that of \cite[\S4.4]{Jacelon:2021vc}, with dimension drop algebras (over the interval) replaced by generalised dimension drop algebras (over $X$), in the sense of \cite{Toms:2007fj, Lin:2010tp}. Specifically, fix $x_0,x_1\in X$ with $d(x_0,x_1)=\diam X$ and for coprime $p,q\in\nn$ define
\[
X_{p,q} = \{f\in C(X,M_p\otimes M_q) \mid f(x_0)\in M_p\otimes 1_q,\: f(x_1)\in1_p\otimes M_q\}.
\]
The $\cs$-algebra $X_{p,q}$ has the following properties. First, since $p$ and $q$ are coprime, $X_{p,q}$ has no nontrivial projections. Next, each trace on $X_{p,q}$ corresponds to some Borel probability measure on $X$, and $\partial_e(T(X_{p,q}))\cong X$ via point evaluations.  Since $X$ is connected, this means that $\hat f\in\aff(T(X_{p,q}))$ is constant for any projection $f\in X_{p,q}$, which in turn means that $X_{p,q}$ has trivial tracial pairing. Nuclearity and the UCT hold for $X_{p,q}$ since it is a type \rm{I} $\cs$-algebra (see \cite[15.8.2, 22.3.5]{Blackadar:1998qf}). Finally, by a suitable interpretation of Stone--Weierstrass (see, for example, \cite{Prolla:1994td}), the Lipschitz elements $\{f\in X_{p,q} \mid \exists K\:\forall x,y\in X\:(\|f(x)-f(y)\|\le Kd(x,y))\}$ are dense in $X_{p,q}$.

Let $(y_m)_{m\in\nn}$ be a dense sequence in $X$. We claim that, for each $m$, there exists a bi-Lipschitz path $\gamma_m\colon[0,1]\hookrightarrow X$ from $x_0$ to $x_1$ such that, for some $t_m\in[0,1]$, $z_m:=\gamma_m(t_m)$ satisfies $d(y_m,z_m)<\frac{1}{m}$. To see this, first note that, by the Hopf--Rinow theorem (see \cite[\S5.3]{Carmo:1976um}), any two points in $X$ can be joined by a length-minimising geodesic, which (when parameterised by arc length) is in the notation of \cite[\S2]{Drutu:2010vl} a `$(1,0)$-quasi-geodesic'. Assume that $y_m\notin\{x_0,x_1\}$, let $\mathfrak{p}_0\colon I\to X$ be a geodesic from $x_0$ to $y_m$ and $\mathfrak{p}_1\colon J\to X$ a geodesic from $y_m$ to $x_1$. If necessary, the path $\mathfrak{p}_1$ can be modified near to $Y:=\mathfrak{p}_0(I)$ so that the concatenation $\mathfrak{p}$ of $\mathfrak{p}_0$ and $\mathfrak{p}_1$ is simple and piecewise bi-Lipschitz. (In a small tubular neighbourhood of  $Y$ (which is diffeomorphic, hence locally Lipschitz equivalent, hence by compactness globally Lipschitz equivalent, to a convex neighbourhood of the normal bundle of $Y$ in $X$) cut out any intermediate points of intersection of $\mathfrak{p}_0$ and $\mathfrak{p}_1$ by going over or around $Y$ via piecewise linear paths in the normal bundle. If $\dim X=1$, then no modification is needed; if $\dim X\ge 3$, there are sufficient dimensions to go up and over $Y$; if $\dim X=2$, we may assume that the geodesic $\mathfrak{p}_0$ is defined on a larger interval $I'\supseteq I$ to give enough room to go around $Y$ at its endpoints.) Applying \cite[Lemma 2.5]{Drutu:2010vl} to $\mathfrak{p}$ yields a simple bi-Lipschitz path $\mathfrak{p}'$ from $x_0$ to $x_1$ at Hausdorff distance less than $\frac{1}{m}$ from $\mathfrak{p}$ (so in particular, there is some point $z_m$ on the path with $d(z_m,y_m)<\frac{1}{m}$). We take this path $\mathfrak{p}'$ as $\gamma_m$ (suitably reparameterised so that its domain is $[0,1]$).

By \cite[Theorem 2.4]{Matouskova:2000ut}, for each $m$ there exists $K_m>0$ and a $K_m$-Lipschitz map $e_m\colon X\to \Gamma_m=\gamma_m([0,1])$ such that $e_m(x_0)=z_m$ and $e_m(x_1)=x_1$ (that is, we extend the Lipschitz function $\gamma_m\circ\max\{t_m,\id\}\circ\gamma_m^{-1}\colon\Gamma_m\to\Gamma_m$ to $X$).

Now we build an inductive limit $\varinjlim(X_{p_m,q_m},\varphi_m)$, with each connecting map $\varphi_m$ of the form $\varphi_m(f)=\ad_u\circ\diag(f\circ\xi_1,\ldots,f\circ\xi_{N_m})$ for some maps $\xi_i\colon X\to X$, most of which are in fact the identity (to get the right trace space in the limit) and very few of which are not $1$-Lipschitz (so that Lipschitz elements in finite stages map to Lipschitz elements in the limit). To accomplish this, let $(p_m,q_m)_{m\in\nn}$ be the sequence of coprime positive integers constructed as in \cite[\S4.4]{Jacelon:2021vc}, ensuring that the numbers $N_m=\frac{p_{m+1}q_{m+1}}{p_mq_m}$ satisfy $\frac{q_{m+1}}{N_m} < \frac{1}{m^2}$ (for tracial control) and $\frac{q_{m+1}}{N_m} < \frac{e^{\frac{1}{m^2}}-1}{K_m}$ (for Lipschitz control). We define the functions $\xi_i\colon X\to X$, $1\le i\le N_m$, by
\[
 \xi_i=	
 \begin{cases}
 \id_X & \text{ if }\ 1\le i\le N_m-q_{m+1}\\
 z_m & \text{ if }\ N_m-q_{m+1}< i \le N_m-q_{m+1}+p_{m+1}\\
 e_m & \text{ if }\ N_m-q_{m+1}+p_{m+1} < i\le N_m.
 \end{cases}
\]
Then, there are unitaries $u_0,u_1\in M_{p_{m+1}}\otimes M_{q_{m+1}}$ (which can be connected via $u\colon X\to\Gamma_m\to[0,1]\to\mathcal{U}(M_{p_{m+1}}\otimes M_{q_{m+1}})$ similar to above) such that $u_0\diag(f\circ\xi_1,\ldots,f\circ\xi_{N_m})u_0^*\in M_{p_{m+1}}\otimes 1_{q_{m+1}}$ and $u_1\diag(f\circ\xi_1,\ldots,f\circ\xi_{N_m})u_1^*\in 1_{p_{m+1}}\otimes M_{q_{m+1}}$. This allows us to define the connecting map $\varphi_m\colon X_{p_m,q_m}\to X_{p_{m+1},q_{m+1}}$ by $\varphi_m(f)=\ad_u\circ\diag(f\circ\xi_1,\ldots,f\circ\xi_{N_m})$.

By construction, $A=\js\otimes\varinjlim(X_{p_m,q_m},\varphi_m)$ has the desired properties. (In fact, if $X$ has finite covering dimension, then by \cite[Theorem 1.6]{Winter:2004kq} and \cite{Winter:2012pi}, $\js$-stability is automatic.) In particular, $A$ is simple (by density of $(z_m)_{m\in\nn}\subseteq X$), $\partial_e(T(A))\cong X$ (since at each stage, most, i.e.\ at least the fraction $1-\frac{1}{m^2}$, of the connecting maps are the identity), and $\widehat{\varphi_{m,\infty}(f)}\in\lip(X,d)$ for every Lipschitz element $f\in X_{p_m,q_m}$ (its Lipschitz constant scaled by at most $\prod\limits_{m\in\nn} e^{\frac{1}{m^2}}$). 
\end{proof}

\begin{remark} \label{remark:sphere1}
\begin{enumerate}[1.]
\item The choice of the base points $x_0,x_1\in X$ in the proof of Theorem~\ref{thm:models} is somewhat arbitrary, but if $X$ happens to be a sphere (with $d$ the geodesic metric), then requiring that $d(x_0,x_1)=\diam X$ (that is, choosing antipodal base points) allows a choice of maps $e_m\colon X\to X$ that are $1$-Lipschitz. Specifically, given $y_m\in X$, let $\gamma_m$ be a geodesic (which we view as the meridian of longitude $0$) from $x_0$ to $x_1$ that passes through $y_m$, and define $e_m\colon X\to \Gamma_m=\gamma_m([0,1])$ by projection onto $\Gamma_m$. With this choice, $\widehat{\varphi_m(f)}\in\lip^1(X,d)$ for every $1$-Lipschitz element $f\in X_{p_m,q_m}$. That said, while potentially important in the one-dimensional setting (see \cite[\S4]{Jacelon:2021vc}), in the present context this is largely an aesthetic observation.
\item The construction is much simpler if we drop the requirement that $A$ be projectionless and is valid for any compact, connected metric space $(X,d)$: $X_{p_m,q_m}$ can be replaced by $A_m=C(X,M_{n_m})$, with $\varphi_m\colon A_m\to A_{m+1}$ defined just in terms of the identity map and point evaluations. In this case, we can arrange as in \cite[Lemma 3.7]{Thomsen:1994qy} to have $\rho_A(K_0(A))=G\cdot1\subseteq\aff(T(A))$ for any prescribed dense subgroup $G$ of $\qq$. In the projectionless case, strict comparison of positive elements (see \cite[Corollary 4.6]{Rordam:2004kq}) implies that $\rho_A(K_0(A))=\zz\cdot1$.  
\end{enumerate}
\end{remark}

\subsection{Finite Rokhlin dimension} \label{subsection:rokhlin}

Since the models constructed in Theorem~\ref{thm:models} are classifiable by the Elliott invariant, we can lift endomorphisms of trace spaces to the $\cs$-level. By ensuring that the lifted action has \emph{finite Rokhlin dimension}, a notion introduced in \cite{Hirshberg:2015wh} and extended in \cite{Hirshberg:2015ty} to actions on not necessarily unital $\cs$-algebras $A$ and further in \cite{Szabo:2019te} to (cocycle) actions of residually finite groups, we can arrange for the crossed product $A\rtimes_\alpha\nn$ to also be classifiable. This is not an essential requirement if our only interest is observing statistical features of the tracial dynamics, but this procedure will be used in \S~\ref{section:range} to investigate the attainable range.

Here, $A\rtimes_\alpha\nn$ is the crossed product by an endomorphism in the sense of Cuntz \cite[\S6.1]{Cuntz:1982uk} (see also \cite{Stacey:1993uc}), namely, $\alpha$ is extended to an automorphism $\underset{\to}{\alpha}$ of
\[
\underset{\to}{A}=\varinjlim\left(\begin{tikzcd}A \arrow[r,"\alpha"] & A \arrow[r,"\alpha"] & A \arrow[r,"\alpha"] & \ldots\end{tikzcd}\right)
\]
and $A\rtimes_\alpha\nn$ is defined to be the corner $p\left(\underset{\to}{A}\rtimes_{\underset{\to}{\alpha}}\zz\right)p$, where $p\in M\left(\underset{\to}{A}\rtimes_{\underset{\to}{\alpha}}\zz\right)$ is the image of $1_{M(A)}$ in the inclusion $M(A) \to M\left(\underset{\to}{A}\right) \to M\left(\underset{\to}{A}\rtimes_{\underset{\to}{\alpha}}\zz\right)$. In particular, $A\rtimes_\alpha\nn$ is stably isomorphic to $\underset{\to}{A}\rtimes_{\underset{\to}{\alpha}}\zz$, which means, for example, that finite nuclear dimension of one is equivalent to finite nuclear dimension of the other (see \cite[Corollary 2.8]{Winter:2010dn}).

For single endomorphisms $\alpha$, Rokhlin dimension is defined in \cite{Hirshberg:2015wh,Hirshberg:2015ty,Szabo:2019te} only when $\alpha$ is invertible, that is, for (cocycle) actions $\zz\to \Aut(A)$. However, just as in \cite[Definition 2.1]{Brown:2014ts}, which covers the case of Rokhlin dimension $0$, the definition admits a natural extension to the noninvertible case. In \cite{Brown:2014ts}, the only tweak is that the finite set $F$ that appears in the definition of Rokhlin dimension is taken to be an arbitrary subset not of $A$ but of $\alpha^p(A)$, where $p\in\nn$ is the integer that specifies the height of the tower. As in \cite[Proposition 2.2]{Brown:2014ts}, this is sufficient to guarantee that the automorphism $\underset{\to}{\alpha}$ has the usual Rokhlin property.

But actually, the only reason for this restriction is to include degenerate (e.g.\ nonunital) examples, especially the shift $\alpha\colon \bigotimes_\nn M_n = M_{n^\infty} \to M_{n^\infty}$, $a_1\otimes a_2\otimes a_3\otimes\cdots \mapsto e\otimes a_1\otimes a_2\otimes a_3\otimes\cdots$ (where $e\in M_n$ is a minimal projection), which yields $M_{n^\infty}\rtimes_\alpha\nn\cong\mathcal{O}_n$ (see \cite[\S2]{Cuntz:1977qy}). Since the unmodified version, that is, \cite[Definition 1.21]{Hirshberg:2015ty} with single noncommuting towers, can reasonably be expected to hold for nondegenerate (e.g.\ unital) endomorphisms, this is the definition we adopt. In other words, in \cite[Definition 1.21]{Hirshberg:2015ty} we drop condition (6) and insist that $f^{(l)}_{1,j}=0$ for all $0\le j\le p$ to arrive at the following.

\begin{definition} \label{def:rok}
An endomorphism $\alpha$ of a $\cs$-algebra $A$ is said to have \emph{Rokhlin dimension} $d$ if $d$ is the least nonnegative integer with the following property. For any finite set $F\subseteq A$, integer $p\ge1$ and $\varepsilon>0$, there are positive contractions $f^{(l)}_{0},\dots,f^{(l)}_{p-1}\in A$, $l\in\{0,1,\dots,d\}$, such that:
\begin{enumerate}[1.]
\item $\|f^{(l)}_{k}f^{(l)}_{j}a\| < \varepsilon$ for every $a\in F$, $l\in\{0,1,\dots,d\}$, $j\ne k\in\{0,1,\dots,p-1\}$;
\item $\left\|\left(\sum_{l=0}^d\sum_{j=0}^{p-1}f^{(l)}_{j}\right)a-a\right\| < \varepsilon$ for every $a\in F$;
\item $\|[f^{(l)}_{j},a]\| < \varepsilon$ for every $a\in F$, $l\in\{0,1,\dots,d\}$, $j\in\{0,1,\dots,p-1\}$;
\item $\left\|\left(\alpha(f^{(l)}_{j})-f^{(l)}_{j+1}\right)a\right\| < \varepsilon$ for every $a\in F$, $l\in\{0,1,\dots,d\}$, $j\in\{0,1,\dots,p-1\}$, where $f^{(l)}_{p}:=f^{(l)}_{0}$.
\end{enumerate}

This is equivalent to the following more succinct version phrased in terms of Kirchberg's central sequence algebra $F(A) = (A_\mathcal{U} \cap A')/\Ann(A,A_\mathcal{U})$ associated to a free ultrafilter $\mathcal{U}$ on $\nn$ (see \cite[\S1]{Kirchberg:2006fk} and \cite[Definition 2.6]{Gardella:2021tb}). Namely, $\alpha$ has Rokhlin dimension $\le d$ if and only if, for every $p$, there are positive contractions $f^{(l)}_{\bar m}\in F(A)$, $l\in\{0,1,\dots,d\}$, $\bar m\in \zz/p\zz$, such that:
\begin{enumerate}[1.] \label{enum:rok}
\item $f^{(l)}_{\bar m}f^{(l)}_{\bar n} = 0$ for every $l\in\{0,1,\dots,d\}$, $\bar m \ne \bar n\in \zz/p\zz$;
\item $\sum_{l=0}^d\sum_{\bar m \in \zz/p\zz} f^{(l)}_{\bar m} = 1$;
\item $\bar\alpha(f^{(l)}_{\bar m}) = f^{(l)}_{\overline{m + 1}}$ for every $l\in\{0,1,\dots,d\}$, $\bar m\in \zz/p\zz$, where $\bar\alpha$ is the action on $F(A)$ induced by $\alpha$.
\end{enumerate}
\end{definition}

Since the Elliott invariant is insensitive to (approximate) unitary equivalence, we obtain the following from \cite{Szabo:2019te} (or in the unital setting, \cite[Theorem 3.4]{Hirshberg:2015wh}). Note that we omit all cocycles and only consider the group $G=\zz$.

\begin{lemma} \label{lemma:rokhlin}
Let $A$ be a separable, $\js$-stable $\cs$-algebra and let $\beta\in\End(A)$ be nondegenerate (that is, $\beta$ maps an(y) approximate unit of $A$ to an approximate unit of $A$). Then, there is a nondegenerate $\alpha\in\End(A)$ such that $\alpha$ has Rokhlin dimension $\le1$ and $\Ell(\alpha)=\Ell(\beta)$.
\end{lemma}

\begin{proof}
By \cite[Remark 11.13]{Szabo:2019te}, there is an automorphism $\theta$ of $\js$ (namely, a tensor product of shifts) that has Rokhlin dimension $1$. Note that $\Ell(\theta)$ is the identity. As in \cite[Theorem 11.5]{Szabo:2019te}, since $\js$ is strongly self-absorbing there is an isomorphism $\varphi\colon A\to A\otimes\js$ that is approximately unitarily equivalent to $\id_A\otimes1_{\js}$. By \cite[Proposition 11.7]{Szabo:2019te}, $\beta\otimes\theta$ has Rokhlin dimension $\le 1$, and hence so does $\alpha=\varphi^{-1}\circ(\beta\otimes\theta)\circ\varphi$.
\end{proof}

Next, we observe that crossed products of classifiable $\cs$-algebras by endomorphisms with finite Rokhlin dimension are also classifiable. 

\begin{lemma} \label{lemma:classifiable}
Let $A$ be a simple, separable $\cs$-algebra of finite nuclear dimension, and let $\alpha$ be a nondegenerate endomorphism of $A$ with finite Rokhlin dimension. Then, $A\rtimes_\alpha\nn$ is simple and has finite nuclear dimension. Moreover, \[T\left(\underset{\to}{A}\rtimes_{\underset{\to}{\alpha}}\zz\right)\cong T\left(\underset{\to}{A}\right)^{\underset{\to}{\alpha}}.\]
\end{lemma}

\begin{proof}
Extending to $\underset{\to}{A}$, we may assume that $\alpha$ is an automorphism. Since $\alpha$ has finite Rokhlin dimension, $\alpha^m$ is outer for every $m\in\zz\setminus\{0\}$. In fact, $\alpha\colon\zz\to\Aut(A)$ is \emph{strongly} outer, that is, for any $m\ne0$ and any $\alpha$-invariant trace $\tau\in T(A)$, the unique extension of $\alpha^m$ to a trace-preserving automorphism $\alpha^m_\tau$ of $\mathcal{M}=\pi_\tau(A)''$, the von Neumann closure of the Gelfand--Naimark--Segal (GNS) representation associated to $\tau$, is outer. To see this, one proceeds exactly as in the proof of \cite[Theorem 7.8 $(3)\Rightarrow(1)$]{Gardella:2021tb}, just replacing $A_\mathcal{U}\cap A'$ by $F(A)$. We recall this argument here for convenience.

\sloppy
By \cite[Propositions 2.2 and 2.3]{Nawata:2019uq}, we have the required (unital, equivariant) quotient map $\kappa_\tau\colon F(A)\to\mathcal{M}^{\mathcal{U}}_\tau\cap\mathcal{M}_\tau'$. The key point is that, were $\alpha^m_\tau$ an inner automorphism of $\mathcal{M}_\tau$, say $\alpha^m_\tau=\ad_u$ for some unitary $u\in \mathcal{M}_\tau$, then every $x\in\mathcal{M}^{\mathcal{U}}_\tau\cap\mathcal{M}_\tau'$ would commute with $u$ and so the action $\overline{\alpha^m_\tau}$ induced by $\alpha^m_\tau$ on $\mathcal{M}^{\mathcal{U}}_\tau\cap\mathcal{M}_\tau'$ would be trivial. We now show that this is not the case. Choose $p$ such that $m\notin p\zz$, and let $f^{(l)}_{\bar k}\in F(A)$, $l\in\{0,1,\dots,d\}$, $\bar k\in \zz/p\zz$, be as in Definition~\ref{def:rok} for this $p$. Since $\kappa_\tau$ is unital, there is $l_0\in\{0,1,\dots,d\}$ such that $\kappa_\tau(f^{(l_0)}_{\bar k}) \ne 0$ for some (and hence, by equivariance, every) $\bar k \in \zz/p\zz$. Then, $\left\{\kappa_\tau(f^{(l_0)}_{\bar k})\right\}_{\bar k \in \zz/p\zz}$ are pairwise orthogonal positive contractions with $\overline{\alpha^m_\tau}(\kappa_\tau(f^{(l_0)}_{\bar k})) = \kappa_\tau(f^{(l_0)}_{\overline{m+k}})$. Since $m\notin p\zz$, there is $\bar k$ such that $\overline{m+k}\ne\bar k$ in $\zz/p\zz$, and so $\overline{\alpha^m_\tau}$ is indeed nontrivial.
\fussy

By \cite[Theorem 3.1]{Kishimoto:1981tg}, outerness of $\alpha\colon\zz\to\Aut(A)$ implies that $A\rtimes_\alpha\zz$ is simple. Finite nuclear dimension is guaranteed by \cite[Theorem 6.2]{Szabo:2019te}. The final statement follows from \cite[Proposition 2.3]{Liao:2016ui}, whose proof also works in the nonunital setting.
\end{proof}

Combining Lemma~\ref{lemma:rokhlin} and Lemma~\ref{lemma:classifiable} yields the following.

\begin{theorem} \label{thm:lift}
Let $A$ be a simple, separable, unital, $\js$-stable $\cs$-algebra that has trivial tracial pairing and satisfies the UCT, and for which the extreme boundary $X_A$ of the trace space $T(A)$ is compact. Then, for every group homomorphism $\kappa_1\colon K_1(A)\to K_1(A)$ and continuous map $h\colon X_A\to X_A$, there exists a unital endomorphism $\alpha$ of $A$ such that $K_1(\alpha)=\kappa_1$, $T(\alpha)|_{X_A}=h$ and the crossed product $A\rtimes_\alpha\nn$ is classifiable.
\end{theorem}

\begin{proof}
By Lemma~\ref{lemma:extend}, we can extend $h$ to a continuous affine map $T(A)\to T(A)$. Then, since $A$ has trivial tracial pairing, $h$, $\kappa_1$ and $\id\colon K_0(A)\to K_0(A)$ determine a homomorphism $\Ell(A)\to\Ell(A)$. By \cite[Theorem 5.12]{Gong:2021va}, there exists a unital endomorphism $\beta\in\End(A)$ such that $T(\beta)=h$ and $K_1(\beta)=\kappa_1$. Finally, by Lemma~\ref{lemma:rokhlin}, we can find $\alpha\in\End(A)$ with $\Ell(\alpha)=\Ell(\beta)$ such that $\alpha$ has finite Rokhlin dimension, which by Lemma~\ref{lemma:classifiable} implies that $A\rtimes_\alpha\nn$ is classifiable.   
\end{proof}

Together, Theorems \ref{thm:models} and \ref{thm:lift} provide the means of lifting a topological dynamical system $(X,h)$ on a compact metric space to a $\cs$-dynamical system $(A,\alpha)$ on a model classifiable $\cs$-algebra, such that the crossed product is also classifiable. Moreover, the statistical phenomena described in \S~\ref{subsection:statistics} can be translated into tracial versions that are witnessed by the dense subset $\lip(A) = \bigcup_{k\in\nn}\lip^k(A)\subseteq A_{sa}$, where
\[
\lip^k(A) = \{a\in A_{sa} \mid \hat a \in \lip^k(X,d)\}.
\]
Here are some illustrative examples. 

\begin{example}[Estimates of large tracial deviation] \label{example:dps}
Suppose that $h\colon X\to X$ is uniquely ergodic. Let $\mu$ be the measure fixed by $h$ and let $\tau_\mu$ be the corresponding unique trace fixed by $\alpha$. Then, by Proposition~\ref{prop:breiman}, for every $\varepsilon>0$ and every $k\in\nn$, there exist constants $c_1,c_2>0$ such that, for every $a\in\lip^k(A)$ and every $n\in\nn$,
\[
\mu\left(\left\{\tau\in \partial_e(T(A)) \mid \left|\frac{1}{n}\sum_{i=0}^{n-1}\tau(\alpha^i(a))-\tau_\mu(a)\right|>\varepsilon\right\}\right) \le c_1e^{-c_2n\varepsilon^2}.
\]
This in particular applies to minimal, uniquely ergodic homeomorphisms of odd spheres. In this case, one computes from the six-term exact sequence that
\[
(K_0(S^{2m-1}_{p,q}),(K_0(S^{2m-1}_{p,q}))_+,[1],K_1(S^{2m-1}_{p,q}))\cong  (\zz,\nn,1,\zz).
\]
Taking $\kappa_1\colon K_1(A)\to K_1(A)$ in Theorem~\ref{thm:lift} to be the zero homomorphism, the extended algebra $\underset{\to}{A}$ has the same Elliott invariant as (and is therefore isomorphic to) a limit of prime dimension drop algebras. One computes from the Pimsner--Voiculescu sequence (see also \cite[Theorem 10.10.4]{Blackadar:1998qf}) that
\[
(K_0(\underset{\to}{A}\rtimes_{\underset{\to}{\alpha}}\zz),(K_0(\underset{\to}{A}\rtimes_{\underset{\to}{\alpha}}\zz))_+,[1],K_1(\underset{\to}{A}\rtimes_{\underset{\to}{\alpha}}\zz)) \cong (\zz,\nn,1,\zz).
\]
Therefore, $\underset{\to}{A}\rtimes_{\underset{\to}{\alpha}}\zz$ is isomorphic (hence $A\rtimes_\alpha\nn$ is at least stably isomorphic) to the $\cs$-algebra $C(Z)\rtimes_\zeta\zz$ of \cite[Proposition 2.8]{Deeley:aa}, which contains $\js$ as a large subalgebra.
\end{example}

\begin{example}[Exponentially fast tracial mixing]
Suppose that $(X,\mu,h)$ has EDC (see Definition~\ref{def:edc}). Then, there exists $\gamma\in(0,1)$ such that for every $a,b\in\lip(A)$ there exists $C>0$ such that for every $n\in\nn$,
\[
\left|\int_{\partial_e(T(A))} \widehat{\alpha^n(a)}\hat b\;d\mu - \tau_\mu(a) \cdot \tau_\mu(b)\right| \le C\gamma^n.
\]
This holds, for example, when the tracial dynamical system is given by a holomorphic automorphism $h$ of positive entropy on a compact K\"{a}hler surface $X$ (see \S~\ref{subsection:examples}).
\end{example}

\begin{example}[The tracial central limit theorem]
Suppose that $(X,\mu,h)$ satisfies the CLT (see Definition~\ref{def:clt}). Then, for every $a\in\lip(A)$ with $\sigma=\sigma_{\hat{a}|_X} > 0$ (which is the typical case) translated so that $\tau_\mu(a)=0$,
\[
\lim_{n\to\infty} \mu\left(\left\{\tau\in \partial_e(T(A)) \mid \frac{1}{\sqrt{n}}\sum_{i=0}^{n-1}\tau(\alpha^i(a)) \le z \right\}\right) = \frac{1}{\sigma\sqrt{2\pi}}\int_{-\infty}^z\exp\left(-\frac{t^2}{2\sigma^2}\right)dt
\]
for every $z\in\rr$. This holds, for example, for tracial dynamical systems given by mixing axiom A diffeomorphisms like hyperbolic toral automorphisms, various systems arising from complex geometry, and certain expanding circle maps (see \S~\ref{subsection:examples}).

\end{example}

\section{The range of the invariant} \label{section:range}

\begin{theorem} \label{thm:range}
Let $\mathcal{K}$ be the class of infinite-dimensional, simple, separable, $KK$-contractible $\cs$-algebras that satisfy the UCT and have continuous scale, nonempty trace space and finite nuclear dimension. Then, for every $B\in\mathcal{K}$ there exists $A\in\mathcal{K}$ and an automorphism $\alpha\in\Aut(A)$ such that $\partial_e(T(A))$ is compact and zero dimensional, and $A\rtimes_\alpha\zz\cong B$.
\end{theorem}

\begin{proof}
By \cite[Theorem 3]{Downarowicz:1991te}, there is a subshift $X$ of the full shift $(\{0,1\}^\zz,h)$ such that the simplex $\mathcal{M}(X)^h$ of $h$-invariant Borel probability measures on $X$ is affinely homeomorphic to $T(B)$. Let $K$ be a Bauer simplex with $\partial_e(K)\cong X$ (namely, the simplex of Borel probability measures on $X$) and let $A$ be the unique object in $\mathcal{K}$ (up to isomorphism) with $T(A)\cong K$. As in the proof of Theorem~\ref{thm:lift}, but appealing to \cite[Theorem 7.5]{Elliott:2020wc} instead of \cite[Theorem 5.12]{Gong:2021va}, we can extend $h$ to an affine homeomorphism $T(A)\to T(A)$ and lift it to an automorphism $\alpha\colon A\to A$ with finite Rokhlin dimension. By Lemma~\ref{lemma:classifiable}, $A\rtimes_\alpha\zz$ is simple and $\js$-stable (since it has finite nuclear dimension---see \cite{Tikuisis:2012kx}), and
\[
T(A\rtimes_\alpha\zz) \cong T(A)^\alpha \cong \mathcal{M}(X)^h \cong T(B).
\]
The Pimsner--Voiculescu sequence shows that $A\rtimes_\alpha\zz$ is $KK$-contractible. The UCT for $A$ passes to the crossed product $A\rtimes_\alpha\zz$  (see \cite[22.3.5]{Blackadar:1998qf}). Finally, $A\rtimes_\alpha\zz$ also satisfies the definition \cite[Definition 5.1]{Elliott:2020vp} of continuous scale: Fix an increasing approximate unit $(e_n)_{n\in\nn}$ for $A$ that satisfies
\[
e_{n+1}e_n = e_ne_{n+1} = e_n \quad \text{ for every } n\in\nn
\]
and such that, for any nonzero positive element $a\in A$, there exists $N\in\nn$ such that
\begin{equation} \label{eqn:comp1}
e_m-e_n\lesssim a \quad \text{ for every } m>n\ge N
\end{equation}
(where $\lesssim$ denotes Cuntz subequivalence). Note that $(e_n)_{n\in\nn}$ passes to an approximate unit for $A\rtimes_\alpha\zz$. Since $\|\tau\|=\lim_{n\to\infty}\tau(e_n)$ for any trace $\tau$, it follows that all traces on $A\rtimes_\alpha\zz$ are bounded. Since $A\rtimes_\alpha\zz$ is simple and $\js$-stable, it has strict comparison of positive elements (see \cite[Theorem 4.4, Theorem 6.6]{Elliott:2009kq}, and note that there are no compact elements of $\cu(A\rtimes_\alpha\zz)$ since $A\rtimes_\alpha\zz$ is stably projectionless). In other words, to demonstrate equation (\ref{eqn:comp1}) for some fixed nonzero positive element $b\in A\rtimes_\alpha\zz$ on the right-hand side, it suffices to show that
\begin{equation} \label{eqn:comp2}
d_\tau(e_m-e_n) < d_\tau(b) \quad \text{ for every } \tau\in T(A\rtimes_\alpha\zz).
\end{equation}
Here, $d_\tau\colon x\mapsto \lim_{n\to\infty}\tau(x^{\frac{1}{n}})$ is the rank function associated to the trace $\tau$. For this fixed $b\in A\rtimes_\alpha\zz$, the function $T(A)^\alpha \to (0,\infty)$, $\tau\mapsto d_\tau(b)$ (first extending $\tau\in T(A)^\alpha$ uniquely to $T(A\rtimes_\alpha\zz)$) is lower semicontinuous, hence attains its nonzero minimum $\varepsilon$ on the compact set $T(A)^\alpha$. By \cite[Theorem 6.6]{Elliott:2009kq}, there exists a positive element $a\in A$ whose Cuntz class corresponds to the constant function $\frac{\varepsilon}{2}$ on $T(A)$, that is, for which $d_\tau(a)=\frac{\varepsilon}{2}$ for every $\tau\in T(A)$. (We may assume that $a$ is in $A$, rather than its stabilisation, by replacing $\varepsilon$ by $\min\{\varepsilon,1\}$. Since for any strictly positive element $h\in A$, the function $\tau\mapsto d_\tau(h)=\|\tau\|$ is constantly $1$ on $T(A)$, this implies that $a\lesssim h$. By \cite[Theorem 11.5]{Elliott:2020vp} or \cite[Corollary 6.8]{Fu:2022aa}, $A$ has stable rank one, so it follows that $a$ is Cuntz equivalent, in fact Murray--von Neumann equivalent, to an element of the hereditary subalgebra generated by $h$, which is $A$.) By equation (\ref{eqn:comp1}), there exists $N\in\nn$ such that, for every $m>n\ge N$ and every $\tau\in T(A)^\alpha$, $d_\tau(e_m-e_n) \le d_\tau(a) < d_\tau(b)$, that is, such that equation (\ref{eqn:comp2}) holds.

We have now verified that $A\rtimes_\alpha\zz\in\mathcal{K}$, and so by the classification obtained in \cite{Elliott:2020wc}, $A\rtimes_\alpha\zz\cong B$.
\end{proof}

\begin{remark} ~\label{remark:final}
\begin{enumerate}[1.]
\item Dropping the assumption of continuous scale, the crossed products of Theorem~\ref{thm:range} cover, up to stable isomorphism, the full class of infinite-dimensional, simple, separable, stably projectionless, $KK$-contractible $\cs$-algebras that have finite nuclear dimension and satisfy the UCT.
\item Suppose that instead of a $KK$-contractible $\cs$-algebra, $A$ is taken to be a limit of prime dimension drop algebras. Then, $B=A\rtimes_\alpha\zz$ has $K$-theory
\[
(K_0(B),(K_0(B))_+,[1],K_1(B)) \cong (\zz,\nn,1,\zz).
\]
Therefore, the $\cs$-algebras constructed in this way include (up to isomorphism) the algebras $C(Z_\varphi)\rtimes_\zeta\zz$ of \cite[\S3]{Deeley:aa}.
\end{enumerate}
\end{remark}

\end{document}